\documentclass[a4paper,12pt]{article}

\usepackage[utf8]{inputenc}
\usepackage[english]{babel}
\usepackage{amsfonts}
\usepackage{amsmath}
\usepackage{amssymb}
\usepackage{amsthm}
\usepackage{enumerate}
\theoremstyle{plain}
\newtheorem{theorem}{Theorem}
\newtheorem{lemma}[theorem]{Lemma}

\newtheorem{corollary}[theorem]{Corollary}

\theoremstyle{definition}
\newtheorem{definition}[theorem]{Definition}

\theoremstyle{remark}

\newtheorem{remark}[theorem]{Remark}

\newcommand*{\1}{1\!\!\,\mathrm{I}}

\newcommand*{\abs}[1]{\left|#1\right|}

\newcommand*{\Var}{\mathrm{Var}\,}

\newcommand*{\cov}[2]{\mathrm{cov}\,\left( #1,#2\right)}
\newcommand*{\ve}{\varepsilon}
\newcommand*{\vph}{\varphi}
\newcommand*{\mbZ}{\mathbb{Z}}
\newcommand*{\mbR}{\mathbb{R}}
\newcommand*{\mbN}{\mathbb{N}}

\newcommand*{\mcB}{\mathcal{B}}
\newcommand*{\mcP}{\mathcal{P}}
\newcommand*{\mcN}{\mathcal{N}}

\newcommand*{\mcF}{\mathcal{F}}
\newcommand*{\mbP}{\mathbf{P}}
\newcommand*{\mbE}{\mathbf{E}}

\begin{document}

\title{GAUSSIAN STRUCTURE IN COALESCING STOCHASTIC FLOWS}

\author{
  A. A. DOROGOVTSEV\\
  \texttt{Institute of Mathematics,}\\
  \texttt{National Academy of Science of Ukraine,}\\ 
  \texttt{Kyiv, Ukraine, andrey.dorogovtsev@gmail.com}
  \and
  E.V. GLINYANAYA\\
  \texttt{Institute of Mathematics,}\\
  \texttt{National Academy of Science of Ukraine,}
  \\ \texttt{Kyiv, Ukraine, glinkate@gmail.com}
}

\maketitle

\begin{abstract}
In the paper we consider the point measure that corresponds to Arratia flow.  The central limit theorem of the multiple integrals with respect to this measure was obtained. 
\end{abstract}
\makeatletter{\renewcommand*{\@makefnmark}{}

%=========================

\section{Introduction.}

Let $\{x(u,t),\ u\in \mbR, \ t\geq 0\}$ be the Arratia flow \cite{Arr}. It is known \cite{Harris} that for every $t>0$  the set $x(\mbR, t)$ is countable and locally finite. Denote $\{u_i\}_{i\in \mbZ} := x(\mbR, t)$ and define the point measure
\begin{equation}
\label{PP}
  N_t = \sum_{i\in \mbZ} \delta_{u_i}.  
\end{equation}

In this paper  we investigate limits in the central limit theorem for  integral functionals with respect to the factorial powers of the point measure $N_t.$ 
\begin{definition}(\cite{LP})
n-th factorial power $N^{(n)}$ of the point measure   $
N = \sum_{i\in \mbZ} \delta_{u_i}$ is defined by
$$
N^{(n)}(C) = \sum_{\substack{ i_1\neq \ldots \neq i_n}} \1{\{(u_{i_1}, \ldots, u_{i_n}) \in C\}},
$$
where $C\in \mathcal{B}(\mbR ^n) .$
\end{definition}
By $N^{\otimes n}$ we denote the usual power of measure  $N,$ i.e. for $\Delta_i \in \mbR,$\\
$$N^{\otimes n}(\Delta_1 \times ... \times\Delta_n) = N(\Delta_1)\ldots N(\Delta_n).$$
\begin{definition}
A locally integrable function $\rho^{(n)}: \mbR^n\to \mbR_+$
is called an $n-$point density function of a  point measure $N$ if, for any  subset $A\in \mathcal{B}(\mbR^n)$
the following formula holds
$$
\mbE N^{( n)}(A) = \int_{A}\rho^{(n)}(\vec x) d\vec x
$$
\end{definition}

For example, the Poisson process on the line with intensity measure $\mu(dx) =p(x)dx$   has $n-$point density function that is equal to  $$\rho^{(n)}(\vec x) = p^{\otimes n}(\vec x)= p(x_1)p(x_2)\ldots p(x_n).$$

In the papers \cite{MRTZ}, \cite{Tribe} the authors proved existence  of the $n-$point density functions $\rho^{(n)}$ for the point measure $N_t.$ 
For instance, 
\begin{align}
\label{densities1}
   &\rho_t^{(1)}(v) = \frac{1} {\sqrt{\pi t}} \\
\label{densities2}
    &\rho_t^{(2)}(v_1, v_2) =\dfrac{1}{\pi t}\left(1+\dfrac{\abs{v_2-v_1}}{2\sqrt t} \cdot e^{-(v_2-v_1)^2/4t} \cdot \int\limits_{\abs{v_2-v_1}/\sqrt t}^{+\infty}e^{-v^2/4}\,dv- e^{-(v_2-v_1)^2/2t}\right)
\end{align}

%\begin{definition}
%A locally integrable function $\rho^{(n)}: \mbR^n\to \mbR_+$
%is called an $n-$point density function of a  point measure $N$ %if, for any  subset\\
%$A\subset \{\vec x\in \mbR^n: \ x_i\neq x_j, \ i\neq j\}$
%the following formula holds
%$$
%\mbE N^{\otimes n}(A) = \int_{A}\rho^{(n)}(\vec x) d\vec x
%$$
%\end{definition}

Connection  between integrals with respect to usual power of measure and factorial power  is given in the next lemma.
\begin{lemma}[\cite{DV}]
\label{exp_of_sum}
For any symmetric   non-negative function $f$ we have
$$
\int_{\mbR^n} f(\vec x)N^{\otimes n}(d \vec x) = \sum_{k=1}^n \sum_{\substack{1\leq l_1\leq \ldots \leq l_k\\ l_1+\ldots+l_k = n}}A_{l_1,\ldots, l_k}^n\int_{\mbR^k}
f(\underbrace{x_1,\ldots, x_1}_{l_1},\ldots, \underbrace{x_k,\ldots, x_k}_{l_k})N^{(k)}(d\vec{x}),
$$

$$
\int_{\mbR^n} f(\vec x)N^{(n)} (d \vec x)= \sum_{k=1}^n \sum_{\substack{1\leq l_1\leq \ldots \leq l_k\\ l_1+\ldots+l_k = n}}a_{l_1,\ldots, l_k}^n\int_{\mbR^k}
f(\underbrace{x_1,\ldots, x_1}_{l_1},\ldots, \underbrace{x_k,\ldots, x_k}_{l_k})N^{\otimes n}(d\vec{x}),
$$
where the constants   $A_{l_1,\ldots, l_k}^n$ and $a_{l_1,\ldots, l_k}^n$ do not depend on $f$ and $N$
\end{lemma}

%\section{ Integral functionals}
 Now we will define integral with respect to the point measure $N_t$ (\ref{PP}) which corresponds to the Arratia flow at time $t$.
\begin{theorem}
The  operator $A_t$
%from $L_2([0,1])$ to $L_2(\Omega, \mcF, \mbP)$
which is defined  for $\phi\in L_2([0,1])$ as
$$
A_t \phi = \int_0^1 \phi(v) dN_t(v)
$$
is a continuous operator acting  from $L_2([0,1])$ to $L_2(\Omega, \mcF, \mbP)$ and
$$
\|A_t\|^2 \leq \frac{1}{\pi t}+\frac{1}{\sqrt{\pi t}}
$$
\end{theorem}
\begin{proof}
For non-negative function $\phi$, by Campbel formula \cite{LP}

$$
\mbE A_t\phi = \int_0^1 \phi(v) \rho^{(1)} (v)dv
$$
and   by  the definition of  $\rho^{(1)} $ and $\rho^{(2)}$
$$
\mbE \left(A_t\phi \right)^2 = \int_{0}^{1}\int_{0}^{1}\phi(v_1) \phi(v_2)
 \rho^{(2)}(v_1, v_2) dv_1 \, dv_2+
\int_{0}^{1}\phi^2(v_1)
 \rho^{(1)}(v_1) dv_1,
$$

For the point measure $N_t$ its $n-$point density functions $\rho_t^{(n)}$ satisfy inequality  $\rho_t^{(n)}\leq \frac{1}{(\pi t)^{n/2}}$  for all $n\geq 1$, \cite{MRTZ}. Using this we can write 
$$
\mbE \left(A_t\phi \right)^2 \leq \left(\frac{1}{\pi t}+\frac{1}{\sqrt{\pi t}}\right)\|\phi\|^2_{L_2([0;1])}
$$

Using obtained inequality, operator $A_t$ can be extended from operator defined on  $C([0;1])$ to a continuous operator from $L_2([0;1])$ to $L_2(\Omega, \mcF, \mbP).$
\end{proof}

\begin{remark}
Note that the variance for the integral  $
A_t \phi = \int_0^1 \phi(v) dN_t(v)
$ is given by formula
\begin{equation}
\label{var}
\Var{A_t \phi}=\\ \int_{0}^{1}\int_{0}^{1}\phi(v_1) \phi(v_2)
 \Big(\rho_t^{(2)}(v_1, v_2)- \frac {1}{{\pi t}} \Big) dv_1 \, dv_2+
\frac {1}{\sqrt{\pi t}}\int_{0}^{1}\phi^2(v_1)
 dv_1
\end{equation}
\end{remark}
From the lemma \ref{exp_of_sum} we get the formulas for moments of the defined integral
\begin{lemma}
\label{moments_int}
$$
\mbE\left(\int_0^1 \phi(u)dN_t(u)\right)^k = \sum_{j=1}^k \sum_{\substack{l_1,\ldots, l_j \geq 1,\\ l_1+\ldots+l_j=k}} A^{k}_{l_1,\ldots,l_j}\int_{\mbR^j} \prod_{i=1}^j \phi^{l_i}(x_i) \rho_t^{(j)}(\vec x) d\,\vec x,
$$
with constants $A^k_{l_1,\ldots,l_j}$ from Lemma \ref{exp_of_sum}.
For example, for $k=3,$
\begin{multline*}
\mbE\left(\int_0^1 \phi(u)dN_t(u)\right)^3 =\\
= \int_{[0;1]^3}\phi^{\otimes 3}(\vec u)\rho_t^{(3)}(\vec u) d\vec u+
3\int_{[0;1]^2}\phi^{\otimes 2}(u_1,u_2)\rho_t^{(2)}(u_1,u_2) \vph(u_1)d\vec u+
\frac{1}{\sqrt{\pi t}}\int_{[0;1]}\phi^3( u)d u
\end{multline*}
\end{lemma}
%\begin{proof}
%The proof is based on the lemma \ref{exp_of_sum}.
%\end{proof}

\begin{remark}
This lemma shows that moments of integrals are finite if integrand has finite moments. 
\end{remark}

Also we can define multiple integral with respect to factorial power of the point measure $N_t$
\begin{theorem}Let $L_{2,{symm}}([0;1]^k)$ be the space of square integrable symmetric functions.
The linear operator $A^k_{t}$ from $L_{2, symm}([0;1]^k)$ to $L_2(\Omega, \mcF, \mbP)$ which is defined   as
$$
A^k_{t} \phi = \int_{[0;1]^k} \phi(\vec v) N^{( k)}_t(d\,v_1,\ldots,d\, v_k),
$$
is a continuous operator.
\end{theorem}
\begin{proof}
By definition of  $k-$point functions $\rho^{(k)}$ of the measure $N_t^{(k)}$
$$
\mbE \int_{[0;1]^k} \phi(\vec v) N^{(k)}_t(d\,v_1,\ldots,d\, v_k) =\int_{[0;1]^k}\phi(\vec v)\rho_t^{(k)}(\vec v) d\, \vec v.
$$
Rewriting the square of integral with respect to factorial product as sum of series indexed with non-equal indexes we get
\begin{multline*}
\mbE \left(\int_{[0;1]^k} \phi(\vec v) N^{( k)}_t(d\,v_1,\ldots,d\, v_k)\right)^2 =\\
=\mbE\sum_{i_1,j_1}\sum_{\substack{i_2\neq i_1\\ j_2\neq j_1}} \ldots
\sum_{\substack{i_k\neq i_1,\ldots, i_{k-1}\\j_k\neq j_1,\ldots, j_{k-1}}} \phi(u_{i_1},\ldots, u_{i_k})\phi(u_{j_1},\ldots, u_{j_k}) = \\
= \sum_{l=0}^k 
c_l\int_{[0;1]^k} \int_{[0;1]^{k-l}}\phi(\vec u)\phi(u_1,\ldots, u_l, v_1, \ldots, v_{k-l})\rho^{(2k-l)}(\vec v,\vec u)d\, \vec v d\,\vec u
\end{multline*}
with some constants $c_l\in \mbN.$

Applying uniform estimation for the $n-$point densities $\rho_t^{(n)}\leq \frac{1}{(\pi t)^{n/2}}$ \cite{MRTZ} we get
$$
\mbE \left(\int_{[0;1]^n} \phi(\vec v) N^{( n)}_t(d\,v_1,\ldots,d\, v_k)\right)^2
\leq C\int_{[0;1]^n} \phi^2(\vec v)d\vec v.
$$
This proves the theorem.
\end{proof}
%\begin{remark}
%Similarly to lemma \ref{moments_int} one can prove that  any moment of the multiple integral with respect to factorial power of point measure is a linear combination of integrals with respect to $n-$point densities.
%\end{remark}

%=========================
\section{ Finite-dimensional central limit theorem for  $\mathbf{(A_t)_ {t>0}}$}

Let $\mcP = \{f:\mbR \to \mbR: \ f|_{[0;1]} \in L_2([0;1]),\ f(x) = f(x+1), \ x \in \mbR\}$ be a class of functions $f$ such that:\\
1) $f$ is 1-periodic;\\
2) restriction of $f$ on the interval $[0;1]$ is square integrable.\\

Similarly to definition of the operator $A_t$ let us define the integral operator $$
A_{k,t}f=\int_k^{k+1}f(u)N_t(du).
$$
Such operators are well-defined for functions $f\in \mcP.$
In this section we will prove  that the sequence $\{A_{k,t}f\}_{k\geq 1}$ satisfies central limit theorem.

The main reason why we can expect the validity of the central limit theorem is the weak dependence of the sequence $ \{A_{k,t} f\}_{k\geq 1}$. It can be checked that this sequence satisfies a mixing condition. 
 %Consider the $\alpha-$mixing coefficient for the Arratia flow $x$
%$$
%\alpha_{x,T}(h)  = \sup\{|\mbP(AB) - \mbP(A) \mbP(B)|, \ A\in \mcF_{-\infty,0}^{x,T}, \ B\in \mcF_{h, +\infty}^{x,T}\}
%$$
%where $\mcF_{a,b}^{x,T}= \sigma\{ x(u, t) \ a<u<b, \ t\in [0,T]\}.$
Recall that for a stationary sequence $\{\xi_k\}_{k\in K}$ (here $K=\mbR \text{ or } K=\mbN$) its $\alpha-$mixing coefficient is defined as
$$
\alpha_{\xi}(h)  = \sup\{|\mbP(AB) - \mbP(A) \mbP(B)|, \ A\in \mcF(\xi)_{-\infty}^0, \ B\in \mcF(\xi)_h^{\infty}\}
$$
where $\mcF(\xi)_{a}^b = \sigma\{\xi_k, \ a<k<b\}.$

\begin{lemma}
\label{mixing}
 Let $f\in \mcP$. For fixed $t>0$ consider sequence
$\xi_k = A_{k,t} f,$ $k\in \mbZ$. Then
$$
\alpha_{\xi}(h) \leq 4\int_{h/3}^{\infty} \frac{1} {\sqrt{2\pi t}} e^{-x^2/2}dx \leq 
\frac {12}{h\sqrt{2\pi t}} \exp(-h^2/18).
$$

\end{lemma}
\begin{proof}
For arbitrary $n, m\in \mathbb N$, denote by $D_{n,m}$ the set  of vectors $\vec u, \vec a \in \mbR^n$, $\vec v, \vec b\in \mbR^m$, such that
$u_1<u_2<\ldots u_n<v_1< v_2<\ldots< v_m$. Consider
\begin{multline*}
\Delta := \sup_{D_{n,m}}|\mbE e^{i(\vec a, x(\vec u,t)) }e^{i(\vec a, x(\vec v,t)) }\1_{\{x(u_n,t)<0\}}\1_{\{x(v_1,t)>h\}}-\\-\mbE e^{i(\vec a, x(\vec u,t)) }\1_{\{x(u_n,t)<0\}}\mbE e^{i(\vec a, x(\vec v,t)) }\1_{\{x(v_1,t)>h\}}|
\end{multline*}
where we denoted by $x(\vec u,t) = (x(u_1,t), \ldots , x(u_n,t)).$ For $L<h/2$ let  $B$ be the random event 
$B = \{x(L,t)>0, x(h-L,t)<h\}$. Denoting by $\tilde x $ an independent copy of the Arratia flow we get
\begin{multline*}
   \Delta\leq  \sup_{D_{n,m}}|\mbE e^{i(\vec a, x(\vec u,t)) }e^{i(\vec a, x(\vec v,t)) }\1_{\{x(u_n,t)<0\}}\1_{\{x(v_1,t)>h\}}\1_B-\\
   -\mbE e^{i(\vec a,  \tilde x(\vec u,t)) }\1_{\{\tilde x(u_n,t)<0\}} e^{i(\vec a,  x(\vec v,t)) }\1_{\{ x(v_1,t)>h\}}| + \mbP (\bar B)\leq \\
   \leq \sup_{D_{n,m}}|\mbE e^{i(\vec a, \tilde x(\vec u,t)) }e^{i(\vec a, x(\vec v,t)) }\1_{\{\tilde x(u_n,t)<0\}}\1_{\{x(v_1,t)>h\}}\1_B-\\
   -\mbE e^{i(\vec a, \tilde x(\vec u,t)) }\1_{\{\tilde x(v_n,t)<0\}} e^{i(\vec a, \tilde x(\vec v,t)) }\1_{\{\tilde x(v_1,t)>h\}}\1_B| + 2\mbP (\bar B),
\end{multline*}
where we use that  the processes $x(u_i, \cdot)$ and $x(v_j, \cdot)$ are independent up to the meeting moment, i.e. $d\langle x(u_i,\cdot), x(v_j,\cdot)\rangle (t) =  \1_{\{x(u_i,t)= x(v_j,t)\}}dt.$
Now, one can get that for $L=h/3$ the upper bound $P(\bar B) \leq 2\int_{h/3}^{\infty} \frac{1} {\sqrt{2\pi t}} e^{-x^2/2}dx $. Since the family of functions $\{e^{i(\vec a, x(\vec u,t))} \1_{x(u_n,t)<0}, \ u_1<u_2<\ldots< u_n, \ n\in \mathbb N\}$ generates the $\sigma-$field $\sigma\{\xi_k, k<0\}$ the statement of the lemma follows from obtained inequality. 
\end{proof}
It is known that if the $\alpha-$mixing coefficient for a stationary sequence decrease fast enough then  under suitable moments conditions this sequence satisfies the central limit theorem \cite {IbLin}. 
\begin{theorem}
\label{CLT}

\begin{enumerate}
  \item Let $f\in \mcP$ such that $f|_{[0;1]} \in L_3([0;1])$. For any $t> 0$ we have
$$
X_{t}^n(f)=\frac {\sum_{k=0}^{n-1}\left(A_{k,t}f - \mbE A_{k,t} f\right)}{\sqrt n} \Rightarrow \zeta_f(t), \ n\to \infty,
$$
where $\zeta_f(t)$ is a Gaussian random variable with zero mean and variance
$$
\sigma^2_t(f) =
\int_{0}^{1}\int_{0}^{1}f(v_1) f(v_2)
 G_t(v_1, v_2)dv_1 \, dv_2
+\frac {1}{\sqrt{\pi t}}\int_{0}^{1}f^2(v_1)dv_1,
$$ with\\
$
    G_t(v_1,v_2)=g_t(v_1-v_2) +  2\sum_{l= 1}^{\infty} g_t(v_1-v_2 +l),
$\\
and
$
    g_t(v_1-v_2) = \rho_t^{(2)}(v_1, v_2)- \frac {1}{{\pi t}}.
$

  \item For any $0<t_1<t_2<\ldots< t_m< T$  and $f \in \mcP$
$$
\Big(X_{t_1}^n(f), \ldots, X_{t_m}^n(f)\Big)\Rightarrow \left(\zeta_f(t_1), \ldots \zeta_f (t_m) \right)
$$
where $\left(\zeta_f(t_1), \ldots \zeta_f (t_m) \right)$ is a Gaussian vector with zero mean and covariance matrix $ \Sigma_{t_1, \ldots, t_m} (f) = (c_{ij})_{i,j = 1}^{m}$ with
$$
c_{ij} =\lim_{n\to \infty}\frac 1 n\cov{ X_{t_{i}}^n(f) }{X_{t_{j}}^n(f) }, \ c_{ii}=\sigma^2_{t_i}(f),
$$
\item For any functions $f^{(1)}, \ldots, f^{(m)}\in \mcP$  and $t>0$ we have
$$
\Big(X_{t}^n\left((f^{(1)}\right), \ldots, X_{t}^n\left(f^{(m)}\right)\Big)\Rightarrow \left(\zeta_{f^{(1)}}(t), \ldots, \zeta_{f^{(m)}}(t)\right), \ n\to \infty,
$$
where $\left(\zeta_t(f^{(1)}), \ldots, \zeta_t(f^{(m)})\right)$ is a centered Gaussian vector with covariance
\begin{multline*}
\cov{\zeta_t(f^{(i)})}{{\zeta_t(f^{(j)})}} =\\=\frac 1 2 \int_0^1 \int_0^1 \left[f^{(i)}(u)f^{(j)}(v)+f^{(i)}(v)f^{(j)}(u)\right]G_t(u,v) dudv +\\+\frac {1}{\sqrt{\pi t}}\int_0^1f^{(i)}(u)f^{(j)}(u)du.
\end{multline*}
\end{enumerate}

\end{theorem}
\begin{remark}
The expression for the covariances $c_{i,l}$ in the second statement of the theorem will be discussed in the next section.
\end{remark}
\begin{proof}

Proof of the first statement of the theorem.\\ For each $t>0$ we denote
$$
X_{t}^n (f)= \frac {\sum_{k=0}^{n-1}\left(A_{k,t}f - \mbE A_{k,t} f\right)}{\sqrt n} =\frac 1 {\sqrt n} \sum_{k=0}^{n-1} Y_k
$$
 Due to periodicity of the function $f$ and  stationarity with respect to spatial variable of the Arratia flow, the sequence of random variables $\{Y_k\}_{k\in \mbZ}$ is strictly stationary.
By lemma \ref{mixing} its mixing coefficient
$$
\alpha_Y(h) \leq  c_T \int_h^{\infty} e^{-y^2/2}dy.
$$
By lemma \ref{moments_int} $\mbE |Y_k|^3<\infty.$  Now we can apply central limit theorem for weakly dependent random variables (\cite {IbLin})
$$
X_{t}^n(f) \Rightarrow \mcN(0, \sigma^2_t(f)),\  \text{as } {n\to \infty,}
$$
where
\begin{multline*}
\sigma^2_t(f) = \lim_{n\to \infty} \frac 1 n \Var{\sum_{k=0}^{n-1} Y_k} = \\
= \lim_{n\to \infty} \frac 1 n \int_{0}^{n}\int_{0}^{n}f(v_1) f(v_2)
 \Big(\rho_t^{(2)}(v_1, v_2)- \frac {1}{{\pi t}} \Big) dv_1 \, dv_2+
\frac {1}{\sqrt{\pi t}}\int_{0}^{n}f^2(v_1)
 dv_1
\end{multline*}
Note that the function $ \Big(\rho_t^{(2)}(v_1, v_2)- \frac {1}{{\pi t}} \Big)$ depends on $v_1, v_2 $ only  via $|v_1-v_2|$ (see formula (\ref{densities2})). Then we can define
$$
g_t(v_1-v_2) := \rho_t^{(2)}(v_1, v_2)- \frac {1}{{\pi t}}.
$$
Using periodicity of $f$
\begin{multline*}
\sigma^2_t(f)
= \lim_{n\to \infty} \frac 1 n \Big[\int_{0}^{n}\int_{0}^{n}f(v_1) f(v_2)
 g_t(v_1-v_2) dv_1 \, dv_2+
\frac {1}{\sqrt{\pi t}}\int_{0}^{n}f^2(v_1) dv_1\Big]
=\\
= \lim_{n\to \infty} \frac 1 n \Big[\sum_{k_1, k_2 = 0}^{n-1}
\int_{k_1}^{k_1+1}\int_{k_2}^{k_2+1}\!\!\!f(v_1) f(v_2)
 g_t(v_1-v_2) dv_1 \, dv_2  + n\frac {1}{\sqrt{\pi t}}\int_{0}^{1}f^2(v_1)dv_1
\Big] = \\
=\lim_{n\to \infty} \frac 1 n \Big[\sum_{k_1, k_2 = 0}^{n-1}
\int_{0}^{1}\int_{0}^{1}f(v_1) f(v_2)
 g_t(v_1+k_1-v_2 -k_2) dv_1 \, dv_2
\Big] + \frac {1}{\sqrt{\pi t}}\int_{0}^{1}f^2(v_1)dv_1 = \\
=\lim_{n\to \infty} \frac 1 n
\int_{0}^{1}\int_{0}^{1}f(v_1) f(v_2)
 \Big[2\sum_{l= 1}^{n-1}(n-l) g_t(v_1-v_2 +l) +  n g_t(v_1-v_2)\Big]dv_1 \, dv_2
+ \\+\frac {1}{\sqrt{\pi t}}\int_{0}^{1}f^2(v_1)dv_1
\end{multline*}

For the function  $g_t(v_1-v_2) := \rho_t^{(2)}(v_1, v_2)- \frac {1}{{\pi t}}$ one can write the precise formula using (\ref{densities2}) :

\begin{align}
\label{gt}
   g_t(v_1-v_2) =\dfrac{1}{\pi t}\left(\dfrac{\abs{v_2-v_1}}{2\sqrt{t}} \cdot e^{-(v_2-v_1)^2/4t} \cdot \int\limits_{\abs{v_2-v_1}/\sqrt t}^{+\infty}e^{-v^2/4t}\,dv- e^{-(v_2-v_1)^2/2t}\right). 
\end{align}

From this it is easy to see that the series $\sum_{l= 1}^{\infty}l g_t(v_1-v_2 +l)$ converges uniformly on  $v_1, v_2 \in [0,1]$, so we can pass to the limit with respect to $n$ in the formula for $\sigma^2_t(f)$

\begin{multline*}
\sigma^2_t(f) =
\int_{0}^{1}\int_{0}^{1}f(v_1) f(v_2)
 \Big[g_t(v_1-v_2) +2\sum_{l= 1}^{\infty} g_t(v_1-v_2 +l) \Big]dv_1 \, dv_2+\\
+\frac {1}{\sqrt{\pi t}}\int_{0}^{1}f^2(v_1)dv_1.
\end{multline*}

Proof of the second  and third statements of the theorem are similar to the first one.
\end{proof}

%=========================
\section{Central limit theorem for conditional expectations.}
Attempt to find covariances in the statement (2) of theorem \ref{CLT} naturally leads to the question of calculation conditional expectations $\mbE \left(\int_{0}^1 f(u) N_{t_2}(du)\big| \mcF_{t_1}^x\right)$.
In this section we study limit behaviour of the conditional expectations for $t_1 <t_2$
$$
\frac 1{\sqrt n} \mbE \left(\int_{0}^n f(u) N_{t_2}(du)\big| \mcF_{t_1}^x\right), \text{ as } n\to \infty
$$
where, as before, $N_t$ is the point measure associated to the Arratia flow $x$ and $\mcF_t^x = \sigma \{ x(u, s),\  u\in \mbR, \ s\leq t\}.$

  To this aim we need representation of the Arratia flow in term of Brownian web.  Recall  that  Brownian web is defined as a family of random processes $\{\varphi_{s,\cdot }(u) \in C([s,\infty))\}$ such that, given $(u_1, t_1), . . . ,(u_m , t_m )$ the processes  $\varphi_{t_1,\cdot }(u_1), . . . , \varphi_{t_m,\cdot }(u_m)$ are coalescing Brownian motions (see, for instance, \cite{BrWeb}). In this section we consider Arratia flow  $x(u,t) = \varphi_{0,t}(u).$ We construct the point measure   $N_{t_2}$  using  measure $N_{t_1}$  and the map $\varphi_{t_1, t_2}$ from Brownian web. 
  
  To describe general properties of such construction let us introduce the family of random measures on $\mbR$ related to the point measure $\mu$ and non-decreasing function $\varphi: \mbR \to \mbR$ such that
\begin{align*}
    &\lim_{x\to +\infty} \varphi(x) = +\infty,\\
    &\lim_{x\to -\infty} \varphi(x) = -\infty.
\end{align*}
For every $k\geq 2$ and vector $\vec v=(v_1,..., v_k)$ denote
\begin{align*}
     &v_* = \min\{v_1,...,v_k\},\\
     &v^* = \max\{v_1,...,v_k\}.
\end{align*}
Define the set
$$
M_k^{\varphi} = \{\vec v: \varphi(v_*)=\varphi(v^*)\} \subset \mbR^k
$$
Note, that $M_k^\varphi$ is a Borel subset of $\mbR^k.$ Denote by  $\mu_\varphi^{(k)}$ the restriction of the measure $\mu^{(k)}$ on the set $M_k^\varphi.$ Finally, define the measure $\mu_{k,\varphi}$ as an image of measure $\mu_\varphi^{(k)}$ under the mapping
$$
\mbR^k \ni \vec v = (v_1,\ldots, v_k)\mapsto \varphi(v_*).
$$
Such  measures can be used for the expression of the counting measure obtained from the measure $\mu$ under the action of the function $\phi.$ Define the counting measure $\nu = (\mu\circ \varphi^{-1})_*$ by 
$$
\nu(\Delta) = |\{u\in \Delta:  \ \mu(\varphi^{-1}\{u\})>0\}|.
$$

The integrals with respect to measure $\nu$ can be expressed in terms of integrals with respect to $\mu_{k,\varphi}, \ k\geq 1.$

\begin{lemma}
\label{mu_k_series}
For a continuous function $f$ on $\mbR$ with compact support
\begin{multline*}
    \int_{\mbR} fd\nu = \int_{\mbR}fd\mu_{1,\vph} -
    \frac{1}{2!} \int_{\mbR} f d\mu_{2,\vph}+
    \frac{1}{3!} \int_{\mbR} f d\mu_{3,\vph}-\ldots+
    (-1)^{n+1} \frac{1}{n!} \int_{\mbR} f d\mu_{n,\vph}+\ldots
\end{multline*}
\end{lemma}
\begin{proof}
Since the measure $\mu $ is locally finite and due to the the condition on the function $\vph$ the series from the right hand side contains only finite number of summands with probability 1.  To prove the lemma it is enough to consider point $y$ such that $\nu(\{y\})=1.$ Then for some ordered points $x_1<\ldots<x_n$ 
$$
\vph(x_1) = \ldots =\vph(x_n) = y, \ \ \mu(\{x_1\}) = \ldots =\mu(\{x_n\}) = 1.
$$
Consequently, 
\begin{multline*}
    \nu(\{y\}) = \Big|\bigcup_{k=1}^n\{\vph(x_k)\}\Big|=
    \sum_{k=1}^n|\{\vph(x_k)\}|-
    \sum_{1\leq k_1<k_2\leq n}|\{\vph(x_{k_1})\}\cap\{\vph(x_{k_2})\}|+\ldots\\
    + (-1)^n\Big|\bigcap_{k=1}^n\{\vph(x_k)\}\Big|=
    \mu_{1,\vph}(\{y\})-\frac{1}{2!}\mu_{2,\vph}(\{y\})+\ldots+\frac{1}{n!}\mu_{n,\vph}(\{y\}).
\end{multline*}
This proves the lemma.
\end{proof}

\begin{remark}
Note, that in case when $\mu$ is generated by the Arratia flow, i.e. $\mu=N_t$, one can check that series from the lemma absolutely converges for any measurable and bounded function $f.$
\end{remark}

In view of the lemma \ref{mu_k_series} we need to prove central limit theorem for the integrals with respect to the measures $\mu_{k,\vph}$ for the case when $\mu = N_{t_1}$ and $\vph_{t_1,t_2}$ is the mapping from the Brownian web. To calculate corresponding mean and variance let us find the expectation of the integrals with respect to measures $\mu_{k,\vph}.$ We will do this in two cases. One when $\mu$ is random and $\vph$ is deterministic and another when $\mu$ is deterministic
 and $\vph$  is random.
 
 \begin{lemma}
 Suppose, that the random measure $\mu$ has the point densities $\rho^{(k)},$ $k\geq 1$ and the function $\vph$ is deterministic. Then 
 $$
 \mbE \int_{\mbR}f d\mu_{k,\vph}= \int_{\mbR^k}f(\vph(v_*))\1_{M^{\vph}_k}(\vec v)\rho^{(k)}(\vec v) d\vec v.
 $$
 \end{lemma}
 \begin{proof}
 The proof follows directly from the definition of measures $\mu_{k,\vph}.$
 \end{proof}

To formulate the statement when $\vph$ is random we need the following notation and assumtion. Suppose that for arbitrary $a<b$ there exist density $q(a,b,u)$ such that for every Borel subset $\Delta \subset \mbR$
$$
\mbP\{\vph(a) = \vph(b), \ \vph(a) \in \Delta\} = \int_{\Delta} q(a,b,u)du.
$$

\begin{lemma}
\label{moment_in_nu}
For the deterministic measure $\mu$ and random function $\vph$ the following relation holds
\begin{multline*}
    \mbE \int_{\mbR} f d\nu = \int_{\mbR}\left(\int_{\mbR} f(v) p(u,v) dv)\right)\mu(du)-\\-
    \frac{1}{2!}\int_{\mbR^2}\left(\int_{\mbR} f(v) q(u_*,u^*,v) dv)\right)\mu^{(2)}(d\vec u)+\ldots\\
    +
      (-1)^{n+1} \frac{1}{n!}\int_{\mbR^n}\left(\int_{\mbR} f(v) q(u_*,u^*,v) dv)\right)\mu^{(n)}(d\vec u)+\ldots.
\end{multline*}
Here $p(u,v)$ is the probability density of $\vph(u),$ $u\in \mbR.$
\end{lemma}
\begin{proof}
The statement of the lemma follows from lemma \ref{mu_k_series}, definitions of measures $\mu_{k,\varphi}$ and functions $q.$
\end{proof}
Now we can consider  summands from the series in the  statement of lemma \ref{moment_in_nu}  with random stationary measure $\mu$.
Recall, that by definition, random measure that is defined on $\mcB(\mbR)$ is stationary if for any number $m\in\mbN,$  of Borel subsets $\Delta_1, \Delta_2, ... \Delta_m \in \mcB(\mbR)$ and any real nuber $h\in \mbR$ the following equality in distribuion holds
$$
(\mu(\Delta_1),\mu(\Delta_2),\ldots, \mu(\Delta_m)) \overset{d}{=}
(\mu(\Delta_1+h),\mu(\Delta_2+h),\ldots, \mu(\Delta_m+h))
$$
To prove the central limit theorem  for integrals
$$\frac 1 n \int_{\mbR^n}\left(\int_0^n f(v) q(u_*,u^*,v) dv)\right)\mu^{(k)}(d\vec u)$$
  let us rewrite them.\\
  For every $k\geq 2$
$$
\int_{\mbR^k}\left( \int_0^nf(v)q(u_*,u^*,v)dv\right)\mu^{(k)}(d\vec u)=
\int_0^n f(v)\left( \int_{\mbR^k}q(u_*,u^*,v)\mu^{(k)}(d\vec u)\right)dv.
$$

Note that for a case when the random mappting $\vph_{t_1,t_2}$ is taken from Arratia web it can be easily checked that
\begin{enumerate}
    \item $\forall z\in \mbR, \ a\leq b, \ v\in\mbR:$
    $$
    q_{t_2-t_1}(a,b,v)=q_{t_2-t_1}(a+z,b+z,v+z),
    $$
    \item
      $$
    q_{t_2-t_1}(a,b,v)\leq p_{t_2-t_1}(a,v)\wedge p_{t_2-t_1}(b,v),
    $$
    where $p_t(a,\cdot)$ is the Gaussian density with mean $a$ and variance $t.$
\end{enumerate}

\begin{lemma}
\label{xi_process}Suppose that $\mu$ is stationary random point measure. 
Then for every $k\geq2$ the random function
$$
\xi_k(v)=\int_{\mbR^k} q_{t_2-t_1}(u_*,u^*,v)\mu^{(k)}(d\vec u), \ v\in \mbR
$$
is a stationary random measure.
\end{lemma}
\begin{proof}
Denote for any arbitrary $z\in\mbR$ the image of the measure $\mu^{(k)}$ under the transformation
$$
\mbR^k \ni (u_1,\ldots, u_k) \mapsto (u_1+z,\ldots, u_k+z)
$$
by $\mu_z^{(k)}.$ Since $\mu$ is a stationary random measure, then for every $z_1,\ldots, z_n$ and $h$ the sequences of random measures
$$
\left\{\mu_{z_1}^{(k)},\ldots, \mu_{z_n}^{(k)}\right\},  
\left\{\mu_{z_1+h}^{(k)},\ldots, \mu_{z_n+h}^{(k)}\right\}
$$
are equidistributed. Now the statement of the lemma follows from the properties of the function $q_{t_2-t_1}.$
\end{proof}
 
 Similarly to mixing coefficients of the stationary process we can define mixing coefficient of the point measure. 
\begin{definition}
A random stationary measure $\mu$  satisfies mixing condition with the function $\alpha$ if for any $t\in \mbR$ and $h>0$
$$
\sup \Big\{\mbP(A\cap B)-\mbP(A)\mbP(B)|, \ A\in \mcF_{-\infty}^t, \ B\in \mcF_{t+h}^{\infty}\Big\}\leq \alpha(h),
$$
where $\mcF_s^t = \sigma\{\mu(\Delta): \Delta\subset [s,t]\}$
\end{definition}

Now we can prove the central limit theorem for the integrals with respect to process $\xi_k$ that was defined in lemma
\ref{xi_process}.
\begin{theorem}
\label{CLT_factorial_measure}
Let random mapping $\vph_{t_1,t_2}$ be an Arratia web and measure $\mu=N_{t_1}$ is the point measure that corresponds to the Arratia flow at the moment of time $t_1$.
Let  $f\in C(\mbR)$ be a 1-periodic function such that 
$\int_{\mbR}f(u)du=0.$ Then  there exists a Gaussian random variable  $\zeta_k(f)$ such that
$$
\frac{1}{\sqrt n}\int_0^n f(v) \xi_k(v)dv \Rightarrow \zeta_k(f), \ n\to \infty,
$$
where 
$$
\xi_k(v)=\int_{\mbR^k} q_{t_2-t_1}(u_*,u^*,v)\mu^{(k)}(d\vec u), \ v\in \mbR.
$$
\end{theorem}
\begin{proof}
Define random variables
$$
\eta_l=\int_l^{l+1} f(v)\xi_k(v) dv, \ l=0,1,\ldots.
$$
This sequence is stationary. Check that $\eta_0$ is centered.
$$
\mbE \eta_0 = \int_0^1 f(v) \mbE \xi_k(v)dv= \mbE\xi_k(0)\int_0^1 f(v)dv=0.
$$
Now check the boundedness of moments. Consider
\begin{multline*}
    \mbE |\int_0^1 f(v) \xi_k(v) dv|^n \leq\\ 
    \leq \int_0^1\ldots \int_0^1 |f(v_1)|\ldots |f(v_n)| 
    \mbE \prod_{j=1}^n \int_{\mbR^k} q_{t_k-t_1}(u_*, u^*, v_j)
    \mu^{(k)} (d\vec u) d\vec v\leq\\
    \leq c_1 \mbE \prod_{j=1}^n \int_{\mbR^k} e^{-c_1\|\vec u \|^2}\mu^{(k)}(d\vec u)<\infty. 
\end{multline*}
The boundedness of the last expression follows from the lemmas about the moments  of integrals with respect to factorial measures and estimations of the $n-$point densities for the point measure of the Arratia flow.
%(!add this lemma separately. this also used in the proof of theorem 7).

To establish now the central limit theorem let us use approximations of the random processes $\xi_k.$ For every $n\geq1$ define 
$$
\xi_k^m(v) = \int_{\mbR^k}q_{t_2-t_1} (u_*, u^*,v)
\1_{\{v-m\leq u_*\leq u^*\leq v+m\}}\mu^{(k)}(d\vec u).
$$
Note that the sequence  of random variables 
$$
\eta_l^m=\int_{l}^{l+1} f(v) \xi_k^m(v) dv,  \ l\geq 0
$$
satisfies mixing condition with the function $\alpha(\cdot -2m),$ where $\alpha$ is the function from mixing condition for measure $\mu.$
Due to estimation of this function for the point measure from the Arratia flow one can conclude that the following convergence holds
$$
\frac{1}{\sqrt n} \sum_{l=0}^{n-1} \eta_l^m \Rightarrow \varkappa_m, \ \ n\to \infty.
$$
Here $\varkappa_m$ is centered Gaussian variable with 
$$
\mbE \varkappa_m^2 = \int_0^1\int_0^1 f(v_1)f(v_2) 
\mbE\prod_{j=1}^2\int_{\mbR^k}q_{t_2-t_1}(u_*, u^*,v_j)\1_{v_j-m\leq u_*\leq u^*\leq v_j+m}\mu^{(k)}(d\vec u).
$$

For  any $ v \in \mbR$ and $m\geq 1$ denote  
$$B_m(v) = \{\vec u \in \mbR^k: \  v-m\leq u_*\leq u^* \leq u+m \}.$$

Then  for every $m\geq 1$ 
\begin{multline*}
    \mbE\left(\frac{1}{\sqrt n} \sum_{l=0}^{n-1} \left(\eta_l-\eta^{m}_l\right)\right)^2=
    \frac 1 n \mbE\left(\int\limits_0^n  \int\limits_{\mbR^k\setminus B_m(v)}f(v) q_{t_2-t_1}(u_*,u^*,v)
    \mu^{(k)}(d\vec u) dv\right)^2=\\
    =\frac 1 n \int\limits_0^n \int\limits_0^n f(v_1)f(v_2) \mbE\int\limits_{\mbR^k\setminus B_m(v_1)}\int\limits_{\mbR^k\setminus B_m(v_2)}q_{s}(u_{1*},u_1^*,v_1)
    q_{t_2-t_1}(u_{2*},u_2^*,v_2)\cdot\\
    \cdot\mu^{(k)}(d\vec u_1) \mu^{(k)}(d\vec u_2) dv_1dv_2.
\end{multline*}

Consider mathematical expectation in the last expression. Using that  
$$    q_{t_2-t_1}(a,b,v)\leq p_{s}(a,v)\wedge p_{s}(b,v),
    $$
    where $p_t(a,\cdot)$ is the Gaussian density with mean $a$ and variance $s,$ and the uniform upper bound on the $k-$point densities of the point measure $\mu$ we get
\begin{multline*}
 \mbE\int\limits_{\mbR^k\setminus B_m(v_1)}\int\limits_{\mbR^k\setminus B_m(v_2)}q_{t_2-t_1}(u_{1*},u_1^*,v_1)
    q_{t_2-t_1}(u_{2*},u_2^*,v_2)
    \mu^{(k)}(d\vec u_1) \mu^{(k)}(d\vec u_2)\leq\\
    \leq C e^{-2\max(2m, |v_1-v_2|)}.
\end{multline*}
Substitute this upper bound into the expression for $ \mbE\left(\frac{1}{\sqrt n} \sum_{l=0}^{n-1} \left(\eta_l-\eta^{m}_l\right)\right)^2$ we get

\begin{multline*}
    \mbE\left(\frac{1}{\sqrt n} \sum_{l=0}^{n-1} \left(\eta_l-\eta^{m}_l\right)\right)^2
  \leq  C \frac 1 n \int\limits_0^n \int\limits_0^n  e^{-\max(2m, |v_1-v_2|)} dv_1dv_2=\\=\frac {2C}{n}\int_0^n e^{-\max(2m, r)}dr \to 0, \text{ as } m\to \infty.
\end{multline*}
From this we conclude that there exists $\sigma^2$ such that $\mbE \varkappa_m^2 \to \sigma^2$ as $m\to \infty$ and 
$$\varkappa_m \Rightarrow \varkappa,$$
where $\varkappa$ is Gaussian random variable with zero mean and variance $\sigma^2.$ Finally, from this follows
$$
\frac{1}{\sqrt n} \sum_{l=0}^{n-1} \eta_l\Rightarrow \zeta_k, \ \ n\to \infty
$$
and this proves the statement of the lemma.
\end{proof}

Putting together lemma (\ref{moment_in_nu}) and theorem (\ref{CLT_factorial_measure}) we get
\begin{theorem}
Let  $f\in C(\mbR)$ be a 1-periodic function such that 
$\int_{\mbR}f(u)du=0$ and $(N_t)_{t\geq 0}$ is the point measure  that corresponds to the Arratia flow $\{x(u,t),\ u\in \mbR, \ t\geq 0\}.$ Then 
for $t_1<t_2$
$$
\frac 1 {\sqrt n} \mbE \left(\int_0^n f(u) N_{t_2}(du)\big| \mcF_{t_1}^x\right) \Rightarrow  \sum_{k=1}^{\infty} (-1)^{k+1}\frac 1 {k!} \zeta_k(f), \text{ as } n\to \infty
$$
where $\zeta_k(f)$ is centered Gaussian variable with 
$$
   \mbE \zeta_k(f)^2 = \int_0^1\int_0^1 \int_{\mbR^k}f(v_1)f(v_2) 
\prod_{j=1}^2q_{t_2-t_1}(u_*, u^*,v_j)\rho_{t_1}^{(k)}(\vec u)(d\vec u)
$$
\end{theorem}

%=========================
\section{ Convergence of integrals in functional spaces.}
Note that the second part of the theorem  \ref{CLT} states that for fixed function $f$ there exists the limiting Gaussian process indexed by time $t\geq 0$. This process defined by it's finite-dimensional distributions with covariance matrix $\Sigma_{t_1,\ldots, t_m}$.
In this section we consider the  sequence of processes  $\{X_t^n, \ t\in [t_0, T]\} $  and study its convergence in the space of continuous functions. The second statement of the theorem \ref{CLT}, gives us the convergence   of finite-dimensional distributions of this sequence to Gaussian distribution. 
%To prove  week convergence of processes in the Skorokhod space we need to control modulus of continuity of the processes $X_{\cdot}^n.$

 In the sequel, together with the Arratia flow we use notion of the dual flow $\{y(u,s), \ u\in \mbR, s\in [0,t]\}$ for the Arratia flow (see, for example, \cite{Dorogovtsev,TothWerner}). The dual flow moves in the backward time and its trajectories do not cross trajectories of the Arratia flow.
The Arratia flow and its dual flow exist as weak limits of scaled random walks moving forward and backward in time, respectively, whose trajectories do not intersect (\cite{Arr}).  The dual  Arratia flow can be  constructed using a Brownian web [11]. Assume that the Arratia flow $x(u,t)$ is a Brownian web that starts from time $s=0,$i.e. $x(u,t)=\varphi_{0,t}(u).$ Then for fixed $t>0$  a dual flow for the the Arratia flow can be defined as follows: for any $s\in [0;t],$
$$
y(u,t-s) = \inf\{\varphi_{r,s}(v): \varphi_{r,t}(v), \ v\in \mathbb Q, \ r\in \mathbb Q\cap [0;t]\} \ \text{a.s.}
$$
Now we find out conditions on a function $f$ under which the processes\\ $\left\{X_t^n(f) = \frac 1 {\sqrt n}\int_0^n f(u)N_t(du), \ t\in [t_0, T] \right\},$ $(n\geq 1)$ have continuous modification.
\begin{lemma}
\label{conv_in_D_1}
Let $f\in C^1(\mbR)$ be a   function such that
$\text{supp}f|_{[0;1]} \subset [\ve, 1-\ve]$ and $\int_0^1 f(u)du = 0.$
Then for every $0<t_0<t_1<t$ and $k= 3,4,...$
$$
\mbE|X_t^1(f) - X_{t_1}^1(f)|^{k}\leq C_{t_0, \ve,k,f} \left[(t - t_1)^{2} + (t - t_1)^{ k}\right],
$$
where the constant $C$ depends on  $t_0,\  \ve,\ k,\ f.$
\end{lemma}
\begin{proof}
Now
\begin{multline*}
X_{t_1}^1(f) - X_{t}^1(f) = \int_0^1 f(u) N_{t_1}(du) - \int_0^1 f(u) N_t(du) = \\
=\int_0^1 \left[ f(u) -f(\varphi_{t_1, t}(u))\right]N_{t_1}(du) - \int_{\mbR\setminus [0;1]} f(\varphi_{t_1,t}(u)) N_{t_1} (du)
\end{multline*}
 The upper bound for $k-$th moment ($k\in \mbZ$) of the  first summand
 \begin{multline*}
\mbE \left|\int_0^1 \left[ f(u) -f(\varphi_{t_1, t}(u))\right]N_{t_1}(du) \right|^{k}\leq\\
\leq \mbE \int_0^n {\substack{k \\ \ldots}} \int_0^n
\prod_{i=1}^k|f(u_i) - f(\varphi_{t_1,t}(u_i))|N_{t_1}(du_1)\ldots N_{t_1}(du_k)=\\
=\mbE \int_0^n {\substack{k \\ \ldots}} \int_0^n
 h(\vec u; t_1, t)N_{t_1}(du_1)\ldots N_{t_1}(du_k),
\end{multline*}
where
 \begin{multline*}
h(\vec u; t_1, t)=\mbE\prod_{i=1}^k|f(u_i) - f(\varphi_{t_1,t}(u_i))|\leq
\prod_{i=1}^k\left(\mbE|f(u_i) - f(\varphi_{t_1,t}(u_i))|^k\right)^{1/k}.
\end{multline*}
Denote by $g(\vec u) =\left(\mbE|f(u) - f(\varphi_{t_1,t}(u))|^k\right)^{1/k}.$ Then we can continue the estimation using lemma \ref{moments_int} and uniform upper bound for $\rho^{(k)}$
\begin{multline*}
\mbE \int_0^n {\substack{k \\ \ldots}} \int_0^n
 h(\vec u; t_1, t)N_{t_1}(du_1)\ldots N_{t_1}(du_k)\leq
 \mbE \left(\int_0^1 g(u; t_1, t)N_{t_1}(du)\right)^k\leq\\
 \leq
 C_{t_1}\int_0^1 g^k(u; t_1, t)du=
 C_{t_1} \int_0^1 \mbE|f(u) - f(u+w(t-t_1))|^k du = \\=
 C_{t_1} \int_0^1 \int_{\mbR}\left(f(u) - f(u+\sqrt{t-t_1}y)\right)^k \frac{1}{\sqrt{2\pi}} e^{-y^2/2}dy du\leq \\
 \leq M{C}_{t_1}\int_0^1 \int_{\mbR}(\sqrt{t-t_1})^k y^k \frac{1}{\sqrt{2\pi}} e^{-y^2/2}dy du=M\widetilde{C}_{t_1}(t-t_1)^{k/2},
\end{multline*}
where $M=\sup_{u\in [0,1]} f'(u)$ and the constant $C_{t_1}\leq \frac {\text{const}}{t_0^{k/2}}$ for $t_0<t_1<t.$
Using this upper bound we get
\begin{multline*}
\mbE\left|\int_0^1\big(f(u)-f(\varphi_{t_1,t}(u))\big)N_{t_1}(du)\right|^{k/2}
\left|\int_0^1\big(f(u)-f(\varphi_{t,t_2}(u))\big)N_{t}(du)\right|^{k/2}\leq\\
\leq
\left(\mbE\left|\int_0^1\big(f(u)-f(\varphi_{t_1,t}(u))\big)N_{t_1}(du)\right|^{k}\right)^{1/2}
\left(\mbE\left|\int_0^1\big(f(u)-f(\varphi_{t,t_2}(u))\big)N_{t}(du)\right|^{k}\right)^{1/2}\leq\\
\leq C\left((t-t_1)^{k/2}(t_2-t)^{k/2}\right)^{1/2}\leq
C(t_2-t_1)^{k/2},
\end{multline*}

Consider now
$$\int_{\mbR\setminus[0;1]} f(\varphi_{t_1,t}(u) N_{t_1} (du)= \int_1^{\infty}f(\varphi_{t_1,t}(u) N_{t_1} (du)+\int_{-\infty}^0f(\varphi_{t_1,t}(u) N_{t_1} (du).$$
Let us estimate the first summand in this expression. The second one can be estimated similarly.
\begin{multline*}
\mbE\left|\int_1^{\infty}f\big(\varphi_{t_1,t}(u)\big)N_{t_1}(du)\right|^k\leq
\mbE \int_1^{\infty} {\substack{k \\ \ldots}} \int_1^{\infty}\prod_{i=1}^k|f\big(\varphi_{t_1,t}(u_i)\big)|N_{t_1}(du_1)\ldots N_{t_1}(du_k)\leq\\
\leq \mbE \int_1^{\infty} {\substack{k \\ \ldots}} \int_1^{\infty}r(\vec u; t_1,t)N_{t_1}(du_1)\ldots N_{t_1}(du_k),
\end{multline*}
where $r(\vec u; t_1, t) = \mbE\prod_{i=1}^k|f\big(\varphi_{t_1,t}(u_i)\big)|.$ Since $\text{supp}f\subset [\ve, 1-\ve]$ and $f$ is bounded we have
\begin{multline*}
r(\vec u, t_1, t) \leq  \left(\prod_{i=1}^k\mbE|f\big(\varphi_{t_1,t}(u_i)\big)|^k\right)^{1/k}
\leq
C\left(\prod_{i=1}^k\mbP\left\{\varphi_{t_1,t}(u_i)\in [\ve, 1-\ve]\right\}\right)^{1/k}\leq\\
\leq C\left(\prod_{i=1}^k\mbP\left\{\sup_{s\in[0;t-t_1]} w(s)>u_i-1+\ve\right\}\right)^{1/k},
\end{multline*}
where $w$ is a Wiener process, $w(0)=0.$
Using  estimation
$$
\mbP\left\{\sup_{s\in[0;t-t_1]} w(s)>u-1+\ve \right\}\leq
\sqrt{\frac {2}{\pi}}\exp\left\{-\frac{(u-1+\ve)^2}{2(t-t_1)}\right\}\frac{\sqrt{t-t_1}}{u-1+\ve}
$$
and lemma \ref{moments_int}
\begin{multline*}
\mbE\left|\int_1^{\infty}f(\varphi_{t_1,t}(u))N_{t_1}(du)\right|^k\leq
 \mbE \int_1^{\infty} {\substack{k \\ \ldots}} \int_1^{\infty}r(\vec u; t_1,t)N_{t_1}(du_1)\ldots N_{t_1}(du_k)\leq\\
\leq\mbE \int_1^{\infty} {\substack{k \\ \ldots}} \int_1^{\infty}\left(\prod_{i=1}^k\sqrt{\frac {2}{\pi}}\exp\left\{-\frac{(u_i-1+\ve)^2}{2(t-t_1)}\right\}\frac{\sqrt{t-t_1}}{u_i-1+\ve}
\right)^{1/k}N_{t_1}(du_1)\ldots N_{t_1}(du_k)\leq\\
\leq C_{t_1} \int_1^{\infty} \sqrt{\frac {2}{\pi}}\exp\left\{-\frac{(u-1+\ve)^2}{2(t-t_1)}\right\}\frac{\sqrt{t-t_1}}{u-1+\ve}
du\leq \text{const} C_{t_1} e^{-\ve^2/2(t-t_1)}(t-t_1)/\ve^2\leq\\
\leq C_{t_0, \ve}(t-t_1)^2
\end{multline*}
\end{proof}

The estimation form the previous lemma allows to apply Kolmogorov theorem. So $\int_0^1 f(u) N_t(du)$ has continuous modification under some conditions on $f.$ To prove that   $\int_0^k f(u) N_t(du)$ has continuous modification we need to get similar estimation for this process. To do this we need following lemma.
\begin{lemma}
\label{conv_in_D_n}
Let $F$ be a function that depends on set of random points $x([0;1],t)$ such that
$$
\mbE F\left(x([0;1],t)\right) = 0  \ \text{ and }  \ \mbE F^8\left(x([0;1],t)\right)<\infty
$$
Then for any $k_1<k_2<k_3<k_4$ the following upper bound holds
\begin{multline*}
\mbE\prod_{i=1}^4 F\left(x([k_i,k_i+1],t)\right)\leq\\ \leq
\left( \mbE F^{4}(x([0;1],t))\right)^{1/2}e^{-(k_2-k_1-1)^2/8t} e^{-(k_4-k_3-1)^2/8t}+\\+
\left( \mbE F^{8}(x([0;1],t))\right)^{1/2}e^{-(k_2-k_1-1)^2/24t}e^{-(k_3-k_2-1)^2/24t} e^{-(k_4-k_3-1)^2/24t}.
\end{multline*}
\end{lemma}
\begin{proof}
At first, consider
\begin{multline*}
\mbE\prod_{i=1}^2 F\left(x([k_i,k_i+1],t)\right)=
\mbE\prod_{i=1}^2 F\left(x([k_i,k_i+1],t)\right)\1_{x(k_1+1,t)<x(k_2,t)}+\\+
\mbE\prod_{i=1}^2 F\left(x([k_i,k_i+1],t)\right)\1_{x(k_1+1,t)=x(k_2,t)}.
\end{multline*}
Let us denote by $\tilde x$ the independent copy of the Arratia flow. Then
\begin{multline*}
\mbE\prod_{i=1}^2 F\left(x([k_i,k_i+1],t)\right)\1_{x(k_1+1,t)<x(k_2,t)}+
\mbE\prod_{i=1}^2 F\left(x([k_i,k_i+1],t)\right)\1_{x(k_1+1,t)=x(k_2,t)}\leq\\
\leq
\mbE F\left(x([k_1,k_1+1],t)\right)F\left(\tilde x([k_2,k_2+1],t)\right)\1_{x(k_1+1,t)<\tilde x(k_2,t)}+\\+
\left[\mbE\prod_{i=1}^2 F^p\left(x([k_i,k_i+1],t)\right)\right]^{1/p} \left[\mbP \{x(k_1+1,t)=x(k_2,t)\}\right]^{1/q}\leq\\
\leq\mbE F\left(x([k_1,k_1+1],t)\right)\mbE F\left(\tilde x([k_2,k_2+1],t)\right)+\\+
2\left[\mbE\prod_{i=1}^2 F^p\left(x([k_i,k_i+1],t)\right)\right]^{1/p} \left[\mbP \left\{\mathcal{N}(0,1)>\frac{k_2-k_1-1}{\sqrt 2}\right\}\right]^{1/q}\leq\\
\leq \frac{2}{\sqrt{2\pi}}\left(\mbE F^{2p}(x([0;1],t))\right)^{1/p}\exp\left\{-\frac{(k_2-k_1-1)^2}{4qt}\right\}\left(\frac{\sqrt{2t}}{k_2-k_1-1}\right)^{1/q}
\end{multline*}

We use similar consideration to obtain upper bound for $\mbE\prod_{i=1}^4 F\left(x([k_i,k_i+1],t)\right).$ 
\begin{multline*}
\mbE\prod_{i=1}^4 F\left(x([k_i,k_i+1],t)\right)=\\
=
\mbE\prod_{i=1}^4 F\left(x([k_i,k_i+1],t)\right)\1_{x(k_2+1,t)<x(k_3,t)}
+\\+\mbE\prod_{i=1}^4 F\left(x([k_i,k_i+1],t)\right)\1_{x(k_2+1,t)=x(k_3,t)}=\\
=\mbE\prod_{i=1}^2 F\left(x([k_i,k_i+1],t)\right)\prod_{i=3}^4 F\left(\tilde x([k_i,k_i+1],t)\right)\1_{x(k_2+1,t)<\tilde x(k_3,t)}
+\\+\mbE\prod_{i=1}^4 F\left(x([k_i,k_i+1],t)\right)\1_{x(k_2+1,t)=x(k_3,t)},
\end{multline*}
where $\tilde x$ is an independent copy of the Arratia flow. Using independency and the previous upper bound with $p=q=2$ we can continue
\begin{multline*}
\mbE\prod_{i=1}^2 F\left(x([k_i,k_i+1],t)\right)\prod_{i=3}^4 F\left(\tilde x([k_i,k_i+1],t)\right)\1_{x(k_2+1,t)<\tilde x(k_3,t)}
+\\+\mbE\prod_{i=1}^4 F\left(x([k_i,k_i+1],t)\right)\1_{x(k_2+1,t)=x(k_3,t)}\leq
\end{multline*}
\begin{multline*}
 \leq
\mbE\prod_{i=1}^2 F\left(x([k_i,k_i+1],t)\right)\mbE\prod_{i=3}^4 F\left(\tilde x([k_i,k_i+1],t)\right)
+\\+\mbE\prod_{i=1}^4 F\left(x([k_i,k_i+1],t)\right)\1_{x(k_2+1,t)=x(k_3,t)}\1_{x(k_1+1,t)<x(k_2,t)}+\\
+\mbE\prod_{i=1}^4 F\left(x([k_i,k_i+1],t)\right)\1_{x(k_2+1,t)=x(k_3,t)}\1_{x(k_1+1,t)=x(k_2,t)}\leq
\\ \leq \left( \mbE F^{4}(x([0;1],t))\right)^{1/2}e^{-(k_2-k_1-1)^2/8t} e^{-(k_4-k_3-1)^2/8t}+\\+
\mbE F(\tilde x([k_1, k_1+1],t)\mbE \prod_{i=2}^4 F(x([k_i,k_i+1],t))+\\
+\mbE\prod_{i=1}^4 F\left(x([k_i,k_i+1],t)\right)\1_{x(k_2+1,t)=x(k_3,t)}\1_{x(k_1+1,t)=x(k_2,t)}\1_{x(k_3+1,t)=x(k_4,t)}
\leq
\end{multline*}
\begin{multline*}
\leq
\left( \mbE F^{4}(x([0;1],t))\right)^{1/2}e^{-(k_2-k_1-1)^2/8t} e^{-(k_4-k_3-1)^2/8t}+\\
+
\left[\mbE\prod_{i=1}^4 F^2\left(x([k_i,k_i+1],t)\right)\right]^{1/2}
\mbP\left\{x(k_{i}+1,t)=x(k_{i+1},t), \ i = 1,2,3 \right\}^{1/2}\leq\\
\leq \left( \mbE F^{4}(x([0;1],t))\right)^{1/2}e^{-(k_2-k_1-1)^2/8t} e^{-(k_4-k_3-1)^2/8t}+\\+
\left( \mbE F^{8}(x([0;1],t))\right)^{1/2}e^{-(k_2-k_1-1)^2/24t}e^{-(k_3-k_2-1)^2/24t} e^{-(k_4-k_3-1)^2/24t}.
 \end{multline*}
 Lemma is proved.
\end{proof}

\begin{theorem}
\label{continuity_inequality}
Let $f\in C^1(\mbR)$ be a  1-periodic function such that
$\text{supp}f|_{[0,1]} \subset [\ve, 1-\ve]$ and $\int_0^1 f(u)du = 0.$ Fix $0<t_0<T<\infty.$ Then for any $s,t$ such that $0<t_0\leq s<t\leq T<\infty$ and for any $n\geq 1$ there exists a constant $C$ which does not depend on $s,t$ and $n$ such that
 $$
\mbE|X^n_t(f)-X^n_s(f)|^4\leq C|t-s|^{2}.
 $$
\end{theorem}
\begin{proof}

 Using representation $X^n_t(f) = \frac 1 {\sqrt n} \sum_{k=0}^{n-1}\int_k^{k+1}f(u)N_t(du)$ we have for any $n \geq 1$ and $s<t$
\begin{multline*}
\mbE(X^n_t(f) - X^n_s(f))^4 = \frac 1 {n^2} \sum_{k_1, k_2,k_3, k_4 = 0}^{n-1} \mbE\prod_{i=1}^4 \left(\int_{k_i}^{k_i+1} f(u)N_t(du)-  \int_{k_i}^{k_i+1} f(u)N_s(du)\right).
\end{multline*}
Using notion of the  dual flow $\psi$ to the Arratia flow,
the integral $\int_{k}^{k+1} f(u)N_t(du)$ can be considered as functional from $\psi_{t,0}([k,k+1]).$ Let us denote by 
$$
F\left(\psi_{t,0}([k,k+1])\right)-F\left(\psi_{s,0}([k,k+1])\right) =\int_{k}^{k+1} f(u)N_t(du)-  \int_{k}^{k+1} f(u)N_s(du).
$$

Since
\begin{multline*}
    \mbE\prod_{i=1}^4 \left(F\left(\psi_{t,0}([k_i,k_i+1])\right)-F\left(\psi_{s,0}([k,k+1])\right)\right)=\\
=\mbE\prod_{i=1}^4\left( F\left(x([k_i,k_i+1],t)\right)-F\left(x([k_i,k_i+1],s)\right)\right)
\end{multline*}

 for arbitrary $k_1<k_2<k_3<k_4$  lemma \ref{conv_in_D_n}  gives the upper bound
\begin{multline*}
\mbE\prod_{i=1}^4\left( F\left(\psi_{t,0}([k_i,k_i+1])\right)-F\left(\psi_{s,0}([k_i,k_i+1])\right)\right)\leq
Ce^{-(k_2-k_1-1)^2/8t} e^{-(k_4-k_3-1)^2/8t}+\\+
Ce^{-(k_2-k_1-1)^2/24t}e^{-(k_3-k_2-1)^2/24t} e^{-(k_4-k_3-1)^2/24t}.
\end{multline*}

Now, using  this upper bound and equality
\begin{multline*}
\sum_{1\leq k_1<k_2<k_3<k_4\leq n} f(k_2-k_1, k_3-k_2, k_4-k_3) = \\=
\sum_{k_1=1}^{n-3}\sum_{h_1=1}^{n-k_1-2}\sum_{h_2=1}^{n-k_1-h_1}\sum_{h_3=1}^{n-k_1-h_1-h_2}f(h_1, h_2, h_3),
\end{multline*}
we get
\begin{multline*}
\frac 1{n^2} \sum_{1\leq k_1<k_2<k_3<k_4\leq n}\mbE\prod_{i=1}^4 \left(F\left(\psi_{t,0}([k_i,k_i+1])\right)-F\left(\psi_{s,0}([k_i,k_i+1])\right)\right)\leq\\
\leq
\left( \mbE F^{4}(\psi_{t,0}([0;1]))-F\left(\psi_{s,0}([k_i,k_i+1])\right)\right)^{1/2}\cdot\\
\cdot\frac 1{n^2}\sum_{k_1=1}^{n-3}\sum_{h_1=1}^{n-k_1-2}
\sum_{h_2=1}^{n-k_1-h_1}\sum_{h_3=1}^{n-k_1-h_1-h_2}
\left(\exp\left\{-(h_1 -1)^2/24t - (h_3-1)^2/24t\right\}\right)\leq \\
\leq C \left( \mbE F^{4}(\psi_{t,0}([0;1]))-F\left(\psi_{s,0}([k_i,k_i+1])\right)\right)^{1/2}.
\end{multline*}

From the Lemma \ref{conv_in_D_1}  we get the upper bound
\begin{multline*}
   \left( \mbE \left|F(\psi_{t,0}([0;1]))-F(\psi_{s,0}([0;1]))\right|^4\right)^{1/2} =\\
   =\left( \mbE\left[ \int_{k}^{k+1} f(u)N_t(du)-  \int_{k}^{k+1} f(u)N_s(du)\right]^4\right)^{1/2} \leq\\
   \leq C[(t-s)^2+(t-s)^4]^{1/2}
\end{multline*}

\end{proof}

\begin{corollary}
Let $f\in C^1(\mbR)$ be a  1-periodic function such that
$\text{supp}f|_{[0,1]} \subset [\ve, 1-\ve]$ and $\int_0^1 f(u)du = 0.$ Then for $0<t_0<T<\infty$ and $n\geq 1$ the random processes 
$\left\{X_t^n(f) = \frac 1 {\sqrt n}\int_0^n f(u)N_t(du), \ t\in [t_0, T] \right\},$ $(n\geq 1)$ have continuous modification
and
$$
\left\{\frac {1}{\sqrt n}\int_0^n f(u) N_t(du), \ t\in [t_0, T]\right\}  \Rightarrow\zeta_f(\cdot)  \ n\to \infty \ \text{in } C([t_0, T]).
$$
\end{corollary}
\begin{proof}
Existence of continuous modifications immediately follows from previous theorem and Kolmogorov theorem.

For the proof of weak convergence we note that Theorem \ref{CLT} stated the finite-dimensional convergence. The tightness  of the sequence $\{X_t^n(f), \ t\in[t_0; T]\}$ follows from theorem \ref{continuity_inequality} (see, for example, Corollary 16.9, \cite{Kallenberg})
\end{proof}

\begin{corollary}
For every $f\in C^1([0;1])$ be a   function such that
$\text{supp} f \subset [\ve, 1-\ve]$ the limiting Gaussian process $\{\zeta_f (t), t\in[t_0,T]\}$ from the theorem \ref{CLT} is continuous.
\end{corollary}
\begin{proof}

In the proof of the previous theorem we obtained the inequality
 $$
\mbE|X^n_t(f)-X^n_s(f)|^4\leq C|t-s|^{2}.
 $$
By Theorem \ref{CLT} weak convergence of finite-dimensional distributions of the sequence $X_n$ holds. By Fatou's lemma
$$
\mbE\liminf_{n\to \infty}|X^n_t(f)-X^n_s(f)|^4\leq \liminf_{n\to \infty} \mbE|X^n_t(f)-X^n_s(f)|^4\leq C|t_2-t_1|^{2}.
$$
From this follows that for the limiting process $\zeta_f(\cdot)$ the same inequality holds
 $$
\mbE|\zeta_f(t)-\zeta_f(s)|^4\leq C|t-s|^{2}.
 $$
 
\end{proof}

%=========================
\section {Limiting Gaussian process as a functional.}

The statement 3 from the Theorem \ref{CLT} give us possibility to define finite-dimensional distributions of the process $\zeta$ as a process indexed by the functions $\vph\in L_3([0;1])$. Namely it was proved that for functions $\vph_{1},...,\vph_{m}\in L_3([0;1]) $ the weak convergence of the random vector $ \left(X_{t}^n(\vph_1),...,X_{t}^n(\vph_m)\right) $ to the random vector
 $\left(\zeta_{\vph_1}(t),\ldots,\zeta_{\vph_m}(t) \right) $ holds. By Kolmogorov theorem one can define the Gaussian random field $\zeta$ on the space of parameters $L_3([0;1])$. Note that the covariance of $\zeta$ is
\begin{multline*}
\cov{\zeta_f(t)}{{\zeta_g(t)}} =\\=\frac 1 2 \int_0^1 \int_0^1  f(u)g(v)(G_t(u,v)+G_t(v,u)) dudv +\\+\frac {1}{\sqrt{2\pi}}\int_0^1f(u)g(u)du=
(f, (1+\tilde G_t)g),
\end{multline*}
where we denoted by $\tilde G_t$ the integral operator in $L_2$ with kernel $\tilde G_t = \frac 1 2 (G_t(u,v)+G_t(v,u)).$
\begin{lemma}
\label{continuity_kernel}
The function  $\tilde G_t$ is continuous.
\end{lemma}
\begin{proof}
From the theorem \ref{CLT}, the function $G_t$ is equal to the series
$$
G_t(v_1,v_2)=g_t(v_1-v_2) +  2\sum_{l= 1}^{\infty} g_t(v_1-v_2 +l),\text{ where } g_t(v_1-v_2) = \rho_t^{(2)}(v_1, v_2)- \frac {1}{{\pi t}}.
$$
Using formula (\ref{gt}) for $g_t$  by Weierstrass criterion for uniform convergence with upper bound\\
$
\dfrac{K+l}{2 \sqrt t} \cdot e^{-(l-K)^2/4t} \cdot \int\limits_{\abs{K+l}/ \sqrt t}^{+\infty}e^{-v^2/4}\,dv+ e^{-(L-K)^2/2t},
$
 the series
$
\sum_{l= 1}^{\infty} g_t(x +l)
$
converges uniformly on the interval $[-K, K]$.
\end{proof}
Since for every fixed $t\geq0$ $\tilde G_t$ is continuous, $\zeta$ can be uniquely expanded to Gaussian random field defined on set of the functions $L_2([0;1])$.

 We will consider $\zeta$ as a generalized random element in Hilbert space $L_2([0;1])$ (\cite{Skor,Dor} ) with the covariance operator $1+\tilde G_t.$
One of the most known examples of generalized Gaussian random element on $L_2$ is a formal derivative $dw$ of the 1-dimensional Wiener process $w(\cdot)$  which often called white noise. Note that for this formal derivative one can consider stochastic integral and multiple stochastic integral for function from $L_2([0;1]^k)$. The definition of the integrals
$$
\int_0^1\int_0^{t_1}\ldots \int_0^{t_{k-1}} f(t_1,t_2,\ldots, t_k)dw(t_1)\ldots dw(t_k)
$$
is obvious.
In the context of our paper one can expect that the multiple integrals with respect to point measure $N_t$ after suitable normalization will converge to some integrals defined by $\zeta.$ First of all let us describe the construction of  multiple integrals with respect to $\zeta.$ In the articles by A. V. Skorokhod \cite{Skor} and by A.A. Dorogovtsev \cite{Dor} it was proposed to describe such integrals in term of the Hilbert-Schmidt forms from generalized Gaussian elements. Let us describe the suitable construction.

Now we will discuss the possibility to define the Hilbert-Schmidt form of the process $\zeta$. 
Let $H$ be a real separable  Hilbert space. Consider generalized Gaussian random element   $\xi$ in $H$ with zero mean and covariance operator $A$. We denote by $H^{\otimes k}_{symm}$ the space of symmetric $k-$linear Hilbert-Schmidt forms of $H$. The Wick product (see, for example, \cite{Dor}) of  Gaussian random variables we denote by $*.$
\begin{lemma}
\label{H-S_form}
Assume that covariance operator $A$ of a generalized Gaussian element $\xi$ is invertible.
Let $A_k \in H^{\otimes k}_{symm}$. Then for every orthonormal basis $\{e_j\}_{j\geq 1}$ in $H$ the series
$$
A_k(\xi, \ldots, \xi):=\sum_{n_1,\ldots, n_k = 1}^{\infty} A(e_{n_1}, \ldots, e_{n_k}) \xi(e_{n_1})*\ldots*\xi(e_{n_k}))
$$
converges in $L_2$ and its value $A_k(\xi, \ldots, \xi)$ does not depend on choice of a basis in $H.$
\end{lemma}

\begin{proof}
Let $A_k$ be a finite-dimensional form, that is there exists a basis $\{e_j\}_{j\geq1}$ in $H$ such that
$
A_k = \sum_{n_1,\ldots, n_k = 1}^{N}a_{n_1, \ldots, n_k} e_{n_1}\otimes\ldots \otimes e_{n_k}.
$
Let us define $\xi_0=A^{-1/2}\xi$ by the rule
$$
\xi_0(e) = \xi(A^{-1/2}e), \ e \in H.
$$
Then for finite-dimensional symmetric form $A_k$
$$
A_k(\xi, \ldots, \xi)= A_k (A^{1/2}\xi_0,\ldots A^{1/2} \xi_0)
$$
and
$$
\mbE A_k(\xi, \ldots, \xi)^2 = k!\|A_k(A^{1/2}\cdot,\ldots A^{1/2}\cdot )\|_k^2,
$$
where $\| \cdot\|_k$ is the Hilbert-Schmidt norm.
Note that
$
\|A_k(A^{1/2}\cdot,\ldots A^{1/2}\cdot )\|_k\leq \|A^{1/2}\|^k \|A_k\|_k.
$
From this upper bound for the norm and the fact that the definition of $A_k(A^{1/2}\cdot,\ldots A^{1/2}\cdot )$ does not depend on choice of basis the statement of the lemma follows.
\end{proof}

From the lemma \ref{continuity_kernel} it follows that the integral  operator  with kernel $\tilde G$ is a Hilbert-Schmidt operator and so it has discrete spectrum with eigenvalues of  finite multiplicity. From this follows  $\dim Ker(1+\tilde G)< +\infty.$
Using the lemma \ref{H-S_form} we can define the action of $k-$linear Hilbert-Schmidt form on limiting Gaussian element $\zeta$ since its covariance operator $1+\tilde G$ is invertible on $Ker^{\bot}(1+\tilde G)$.

Next theorem gives us the limit of multiple integrals. Before we formulate it we need to define a $k-$linear form associated to a function $f.$
Let  $f: \mbR^k \to \mbR$ be a symmetric function from $L_2([0;1]^k).$ Then for an orthogonal basis
$\{e_j\}_{j\geq 1}$ in $L_2([0;1])$ it can be written as the series
$$
f= \sum_{i_1,\ldots, i_k} a_{i_1,\ldots, i_k} e_{i_1}\otimes\ldots \otimes e_{i_k}.
$$
We denote by $A_f$ the $k-$linear Hilbert-Schmidt form  on $L_2([0;1])$ which is represented with the same series
$$
A_f= \sum_{i_1,\ldots, i_k} a_{i_1,\ldots, i_k} e_{i_1}\otimes \ldots \otimes e_{i_k},
$$
that is
$A_f(h_1,\ldots, h_k) =\sum_{i_1,\ldots, i_k} a_{i_1,\ldots, i_k} (e_{i_1},h_1) \ldots (e_{i_k},h_k). $
As was mentioned above the action of the form $A_f$ on the generalized Gaussian element $\zeta$  is well defined. Also we denote by
$$
A_f(\zeta, \ldots, \zeta, \delta_{x_j},\ldots, \delta_{x_k}) = \sum_{i_1,\ldots, i_k} a_{i_1,\ldots, i_k} \zeta(e_{i_1})*\ldots *\zeta(e_{i_{j-1}}) e_{i_{j}}(x_j)\ldots e_{i_k}(x_k).
$$
In the term of multi-linear forms from Gaussian element $\zeta$ we can describe the limit distribution of the multiple integral with respect to the point measure of the Arratia flow $N_t.$
\begin{theorem}
Let $f:\mbR^k \to \mbR$ be a symmetric periodic with period 1  function such that
$f|_{[0;1]^k} \in L_2([0;1]^k)$ and $\int_0^1 f(\vec x) dx_j = 0,$ $ j=1,\ldots, k$. Then for $k\in 2\mbZ$
\begin{multline*}
    \frac 1 {n^{k/2}} \int_0^n {\substack{k \\ \ldots}} \int_0^n f(\vec x)N_t^{\otimes k}(dx_1\ldots dx_k) \Rightarrow A_f(\zeta, \ldots, \zeta) + \\+ \int_0^1\int_0^1 C_k^2 A_f(\zeta,\ldots, \zeta, \delta_{x_{k-1}}, \delta_{x_k}) G_t(x_{k-1}, x_k) dx_k dx_{k-1}\ldots+\\+
    \ldots \int_0^1{\substack{ \\ \ldots}}\int_0^1  \frac {k!}{(k/2)! 2^{k/2}}A_f(\delta_{x_1},\ldots \delta_{x_k}) G_t(x_1,x_2)\ldots G_t(x_{k-1}, x_k) d\vec x,\ n \to \infty
\end{multline*}
and for $k\in \mbZ \setminus 2\mbZ$
\begin{multline*}
    \frac 1 {n^{k/2}} \int_0^n {\substack{k \\ \ldots}} \int_0^n f(\vec x)N_t^{\otimes k}(dx_1\ldots dx_k) \Rightarrow A_f(\zeta, \ldots, \zeta) + \\+ \int_0^1\int_0^1 C_k^2 A_f(\zeta,\ldots, \zeta, \delta_{x_{k-1}}, \delta_{x_k}) G_t(x_{k-1}, x_k) dx_k dx_{k-1}\ldots+\\+
    \ldots \int_0^1{\substack{ \\ \ldots}}\int_0^1  \frac {k!}{((k-1)/2)! 2^{(k-1)/2}}A_f(\zeta, \delta_{x_2},\ldots \delta_{x_k}) G_t(x_2,x_3)\ldots G_t(x_{k-1}, x_k) d\vec x,\ n \to \infty
\end{multline*}
\end{theorem}
\begin{proof}
We prove the statement using mathematical induction with respect to dimension $k$. Consider $k =2$. Let $\{e_j\}_{j\geq 0}$ be a orthonormal basis in $L_2([0;1])$ such that $\int_0^1 e_i(x)dx=0$ for                $i\geq 1$ and $e_0 = 1.$ Then function can be represented as
$f|_{[0;1]^2}= \sum_{i,j=1}^{\infty} c_{ij} e_i\otimes e_j.$ Denote by $\vph_j$ the periodic extension on $\mbR$ of the function $e_j$. Then we have
\begin{multline*}
\int_0^n\int_0^nf(x,y)N_t^{\otimes2}(dxdy) = \int_0^n \int_0^n f(x,y)N_t(dx)N_t(dy) - \int_0^n f(x,x)N_t(dx) = \\
=\sum_{i,j}^{\infty}c_{ij}\int_0^n \vph_i(x)N_t(dx)\int_0^n \vph_j(y)N_t(dy)- \int_0^n f(x,x)N_t(dx).
\end{multline*}
Using theorem \ref{CLT} we have
$$
\frac 1 n \int_0^n \vph_i(x)N_t(dx) \int_0^n \vph_j(y)N_t(dy)\Rightarrow \zeta(e_i)\zeta(e_j), \ n\to \infty
$$
and $\frac 1 n  \int_0^n f(x,x)N_t(dx) \Rightarrow \frac 1{\sqrt{\pi t}}\int_0^1 f(x,x)dx,$ $n\to \infty.$
From this follows
\begin{multline*}
\frac 1 n \int_0^n\int_0^nf(x,y)N_t^{\otimes2}(dxdy) =\\
=\sum_{i,j}^{\infty}c_{ij}\frac 1 n \int_0^n \vph_i(x)N_t(dx)\int_0^n \vph_j(y)N_t(dy)- \frac 1 n\int_0^n f(x,x)N_t(dx) \Rightarrow \\
\Rightarrow \sum_{i,j}^{\infty}c_{ij}\zeta( e_i)\zeta(e_j)-\frac 1{\sqrt{\pi t}}\int_0^1 f(x,x)dx, \ n\to \infty.
\end{multline*}
By definition of Wick product  the sum in the last expression is written as
\begin{multline*}
\sum_{i,j}^{\infty}c_{ij}\zeta( e_i)\zeta(e_j) -\frac 1{\sqrt{\pi t}}\int_0^1 f(x,x)dx =\\=
 \sum_{i,j}^{\infty}c_{ij}(\zeta( e_i)*\zeta(e_j) + \frac 1 2 \int_0^1\int_0^1 e_i(x)e_j(y)(G_t(x,y) + G_t(y,x))dxdy = \\= A_f(\zeta, \zeta)+ \int_0^1\int_0^1 f(x,y)G_t(x,y)dxdy,
\end{multline*}
where in the last equality we used symmetry of function $f.$

Now assume that we know the statement of theorem for $k-1$. To do the step of induction we rewrite the $k-$multiple integral with respect to $N^{\otimes k}$ in therm of $(k-1)-$multiple integral. Since a basis of $L_2$ can be taken  with step functions $e_j$, we can write
\begin{multline*}
    \int_0^1 {\substack{k\\ \ldots}}\int_0^1 e_1(x_1)\ldots e_k(x_k) N^{\otimes k}(dx_1\ldots dx_k) =\\=
    \sum_{i_1}\sum_{i_2\neq i_1}\ldots \sum_{i_k\neq i_{k-1},\ldots, i_1} e_1(u_{i_1})\ldots e_k(u_{i_k})=\\=
    \sum_{i_1}\sum_{i_2\neq i_1}\ldots \sum_{i_{k-1}\neq i_{k-2},\ldots, i_1} \left(\sum_{i_k} e_1(u_{i_1})\ldots e_k(u_{i_k}) -
    \sum_{j=1}^{k-1} e_1(u_{i_1})\ldots e_k(u_{i_j})\right)=\\=
    \int_0^1 {\substack{k-1\\ \ldots}}\int_0^1 e_1(x_1)\ldots e_{k-1}(x_{k-1}) N^{\otimes k-1}(dx_1\ldots dx_{k-1}) \int_0^1 e_k (x)N(dx)-\\-
    \sum_{j=1}^{k-1}\int_0^1 {\substack{k-1\\ \ldots}}\int_0^1 e_1(x_1)\ldots e_{k-1}(x_{k-1})e_k(x_j) N^{\otimes k-1}(dx_1\ldots dx_{k-1}).
\end{multline*}

Using the limit behaviour for $(k-1)-$multiple integral we get the statement of the theorem.

\end{proof}

\begin{corollary}For the function $f$ satisfying conditions of the previous theorem and off-diagonal (i.e.$f(\vec x) = 0$ if $x_i=x_j$ for some $i\neq j $) we have
$$
\frac 1 {n^{k/2}} \int_0^n {\substack{k \\ \ldots}} \int_0^n f(\vec x)N^{\otimes k}(dx_1\ldots dx_k) \Rightarrow A_f(\zeta, \ldots, \zeta), \ n\to \infty.
$$
\end{corollary}


\begin{thebibliography}{0}

\bibitem{Arr}
R. Arratia  Brownian motion on the line / R. Arratia // PhD dissertation, Univ. Wisconsin -- 1979. -- 128 p.

\bibitem{Bill}
P. Billingsley, \textit{Convergence of probability measures},
Wiley Series in Probability and Statistics, New York, Second edition, (1999)


\bibitem{DV}
D. J. Daley and D. Vere-Jones, \textit{An introduction to the theory of point processes,} SpringerVerlag, New York 1988

\bibitem{Dor}
A.A. Dorogovtsev  \textit{Stochastic analysis and random maps in Hilbert space}. – VSP, Utrecht, 1994. – ii+109 pp.

\bibitem{Dorogovtsev}
A. A. Dorogovtsev,  and I. I. Nishchenko, An analysis of stochastic flows \textit{Communications on Stochastic Analysis:} Vol. 8 : No. 3 , Article 4, 2014

\bibitem{BrWeb}
 L.R.G. Fontes, M. Isopi, C.M. Newman,   K. Ravishankar, The Brownian web:Characterization and convergence, \textit{The Annals of Probability,} 32(4): 2857-2883, (2004)

\bibitem{Glin_ergodicity}
E.V. Glinyanaya,  Spatial Ergodicity of the Harris Flows, \textit{Communications on Stochastic Analysis:} Vol. 11 : No. 2(2017)

\bibitem{Harris}
T.~E. Harris, 
Coalescing and noncoalescing stochastic flows in $R_1$,
\textit{Stochastic Processes and their Applications}~\textbf{17} (1984) 187--210.


\bibitem{IbLin}
I. A. Ibragimov,  and Yu. V. Linnik,   \textit{ Independent and Stationary
Sequences of Random Variables} (edited by J. F. C. Kingman). Groningen:
Wolters-Noordhof, 1971.


\bibitem{Kallenberg}
O. Kallenberg,  \textit{
	Foundations of Modern Probability}, Probability and Its Applications, Springer Science \& Business Media, 2002



\bibitem{LP}
G. Last,  M. Penrose, \textit{Lectures on the Poisson Process} (Institute of Mathematical Statistics Textbooks). Cambridge: Cambridge University -- 2017

\bibitem{MRTZ}
 R. Munasinghe, R. Rajesh, R. Tribe, O. Zaboronski, Multi-Scaling of then-Point Density Function
for Coalescing Brownian Motions,\textit{ Communications in Mathematical Physics} 268(3):717-725, 2006


\bibitem{Skor}
A. V. Skorokhod, On a Generalization of a Stochastic Integral, \textit{Theory of Probability and its Applications}, 1976, 20:2, 219–233

\bibitem{TothWerner}
B. T\'oth, W. Werner, 
The true self-repelling motion, 
\textit{Probability Theory and Related Fields}~\textbf{111} (1998) 375--452.


\bibitem{Tribe}
R. Tribe, O. Zaboronski, Pfaffian formulae for one dimensional coalescing and annihilating systems,\textit{ Electronic Journal of Probability}, vol. 16, Article 76 (2011)



\end{thebibliography}
\end{document}